\documentclass[12pt]{article}
\usepackage[utf8]{inputenc}
\usepackage[T1]{fontenc}
\usepackage{url}
\usepackage{graphicx}
\usepackage{amsmath}
\usepackage{amsfonts}
\usepackage{amsthm}
\usepackage[margin=1in]{geometry}

\newcommand{\ud}{\,\mathrm{d}}
\numberwithin{equation}{section}
\begin{document}
\date{}

\title{\huge{\textbf{Scaling laws and the rate of convergence in thin magnetic films}}}
\author{\Large{Davit Harutyunyan}\\
\textit{University of Utah}}

\maketitle

\textbf{Abstract.} We study static 180 degree domain walls in thin infinite magnetic films. We establish the scaling of the minimal energy by $\Gamma$-convergence and the energy minimizer profile, which turns out to be the so called \textit{transverse wall} as predicted in earlier numerical and experimental work. Surprisingly, the minimal energy decays faster than the area of the film cross section at an infinitesimal cross section diameter.
We establish a rate of convergence of the rescaled energies as well.

\tableofcontents

\section{Introduction}

In the theory of micromagnetics the energy of micromagnetics is given by
$$E(m)=A_{ex}\int_\Omega|\nabla m|^2+K_d\int_{\mathbb R^3}|\nabla u|^2+Q\int_\Omega\varphi(m)-2\int_\Omega H_{ext}\cdot m,$$
where $\Omega\in \mathbb R^3$ is a region occupied by a ferromagnetic body, $m\colon\Omega\to\mathbb S^2$ with $m=0$ in $\mathbb R^3\setminus \Omega$ is the magnetization vector, $A_{ex},$ $K_d$, $Q$ are material parameters, $H_{ext}$ is the externally applied magnetic field, $\varphi$ is the anisotropy energy density and $u$ is induced field potential, obtained from Maxwell's equations of magnetostatics,
$$
\begin{cases}
\mathrm{curl} H_{ind}=0 & \quad\text{in}\quad \mathbb R^3\\
\mathrm{div}(H_{ind}+m)=0 & \quad\text{in}\quad \mathbb R^3,
\end{cases}
$$
where $H_{ind}=\nabla u.$
Namely $u$ is a weak solution of
$$\triangle  u= \mathrm{div} m\qquad \text{in}\qquad \mathbb R^3.$$

According to the theory of micromagnetics, stable magnetization patterns are described by the minimizers of the micromagnetic energy functional, e.g. [\ref{bib:Hu.Sc},\ref{bib:De.Ko.Mu.Ot1},\ref{bib:De.Ko.Mu.Ot2},\ref{bib:De.Ko.Mu.Ot.Sch}]. In the resent years the study of thin structures in  micromagnetics, in particular thin films and wires have been of great interest, see [\ref{bib:At.Xi.Fa.Al.Pe.Co},\ref{bib:Be.Ni.Kn.Ts.Er},\ref{bib:Ca.La},\ref{bib:Fo.Sc.Su.Sc.Ts.Di.Fi},\ref{bib:K.Kuhn},\ref{bib:Na.Th},\ref{bib:Ni.We.Ba.Ki.Go.Fi.Kr},\ref{bib:Pi.Ge.De.Le.Fe.Le.Ou.Fe},\ref{bib:SS},\ref{bib:Wi.No.Us}] for nanowires and
[\ref{bib:Ca.Me.Ot},\ref{bib:De.Ko.Mu.Ot2},\ref{bib:De.Ko.Mu.Ot.Sch},\ref{bib:Ga.We},\ref{bib:Ko.Sl},\ref{bib:Na.Th}].
It was suggested in [\ref{bib:At.Xi.Fa.Al.Pe.Co}] that magnetic nanowires can be used as storage devices. It is known that the magnetization pattern reversal time is closely related to the writing and reading speed of such a device, thus it has been suggested to study the magnetization reversal and switching processes. In [\ref{bib:Fo.Sc.Su.Sc.Ts.Di.Fi}] the magnetizetion reversal process has been studied numerically in cobalt nanowires by the Landau-Lishitz-Gilbert equation. In thin wires the transverse mode has been observed: the magnetization in almost constant on each cross section forming a domain wall that propagates along the wire, while in relatively thick wires the vortex wall has been observed: the magnetization is approximately tangential to the boundary and forms a vortex which propagates along the wire. In [\ref{bib:He.Ki}] similar study has been done for thin nickel wires and the same results have been observed. When a homogenous external field is applied in the axial direction of the wire facing the homogenous magnetization direction, then at a critical strength the reversal of the magnetization typically starts at one end of the wire creating a domain wall, which moves along the wire. The domain wall separates the reversed and the not yet reversed parts of the wire. In [\ref{bib:Ca-Al.Otto}] Cantero-Alvarez and Otto considered the problem of finding the scaling of critical field in terms of the thin film cross section and material parameters. The authors found four different scaling and corresponding four different regimes. In Figure 1 one can see the transverse and the vortex wall longitudinal and cross section pictures for wires with a rectangular cross section. \\

\setlength{\unitlength}{1mm}
\begin{picture}(180,93)

\put(0,15){\line(0,1){71}}
\put(0,15){\line(1,0){20}}
\put(20,15){\line(0,1){71}}
\put(0,86){\line(1,0){20}}

\put(2,15){\vector(0,1){4}}
\put(5,15){\vector(0,1){4}}
\put(8,15){\vector(0,1){4}}
\put(11,15){\vector(0,1){4}}
\put(14,15){\vector(0,1){4}}
\put(17,15){\vector(0,1){4}}

\put(2,20){\vector(0,1){4}}
\put(5,20){\vector(0,1){4}}
\put(8,20){\vector(0,1){4}}
\put(11,20){\vector(0,1){4}}
\put(14,20){\vector(0,1){4}}
\put(17,20){\vector(0,1){4}}

\put(2,25){\vector(1,3){1.2}}
\put(5,25){\vector(1,3){1.2}}
\put(8,25){\vector(1,3){1.2}}
\put(11,25){\vector(1,3){1.2}}
\put(14,25){\vector(1,3){1.2}}
\put(17,25){\vector(1,3){1.2}}

\put(2,29){\vector(1,3){1.2}}
\put(5,29){\vector(1,3){1.2}}
\put(8,29){\vector(1,3){1.2}}
\put(11,29){\vector(1,3){1.2}}
\put(14,29){\vector(1,3){1.2}}
\put(17,29){\vector(1,3){1.2}}

\put(2,33){\vector(1,2){1.8}}
\put(5,33){\vector(1,2){1.8}}
\put(8,33){\vector(1,2){1.8}}
\put(11,33){\vector(1,2){1.8}}
\put(14,33){\vector(1,2){1.8}}
\put(17,33){\vector(1,2){1.8}}

\put(2,37){\vector(1,1){3.6}}
\put(5,37){\vector(1,1){3.6}}
\put(8,37){\vector(1,1){3.6}}
\put(11,37){\vector(1,1){3.6}}
\put(14,37){\vector(1,1){3.6}}
\put(17,37){\vector(1,1){3.6}}

\put(2,41){\vector(2,1){3.6}}
\put(5,41){\vector(2,1){3.6}}
\put(8,41){\vector(2,1){3.6}}
\put(11,41){\vector(2,1){3.6}}
\put(14,41){\vector(2,1){3.6}}
\put(17,41){\vector(2,1){3.6}}

\put(2,43.5){\vector(3,1){3.6}}
\put(5,43.5){\vector(3,1){3.6}}
\put(8,43.5){\vector(3,1){3.6}}
\put(11,43.5){\vector(3,1){3.6}}
\put(14,43.5){\vector(3,1){3.6}}
\put(17,43.5){\vector(3,1){3.6}}

\put(2,45.5){\vector(4,1){4}}
\put(5,45.5){\vector(4,1){4}}
\put(8,45.5){\vector(4,1){4}}
\put(11,45.5){\vector(4,1){4}}
\put(14,45.5){\vector(4,1){4}}
\put(17,45.5){\vector(4,1){4}}

\put(1,47.5){\vector(1,0){4}}
\put(6,47.5){\vector(1,0){4}}
\put(11,47.5){\vector(1,0){4}}
\put(16,47.5){\vector(1,0){4}}

\put(2,49.5){\vector(4,-1){4}}
\put(5,49.5){\vector(4,-1){4}}
\put(8,49.5){\vector(4,-1){4}}
\put(11,49.5){\vector(4,-1){4}}
\put(14,49.5){\vector(4,-1){4}}
\put(17,49.5){\vector(4,-1){4}}

\put(2,51.5){\vector(3,-1){3.6}}
\put(5,51.5){\vector(3,-1){3.6}}
\put(8,51.5){\vector(3,-1){3.6}}
\put(11,51.5){\vector(3,-1){3.6}}
\put(14,51.5){\vector(3,-1){3.6}}
\put(17,51.5){\vector(3,-1){3.6}}

\put(2,54){\vector(2,-1){3.6}}
\put(5,54){\vector(2,-1){3.6}}
\put(8,54){\vector(2,-1){3.6}}
\put(11,54){\vector(2,-1){3.6}}
\put(14,54){\vector(2,-1){3.6}}
\put(17,54){\vector(2,-1){3.6}}

\put(2,58){\vector(1,-1){3.6}}
\put(5,58){\vector(1,-1){3.6}}
\put(8,58){\vector(1,-1){3.6}}
\put(11,58){\vector(1,-1){3.6}}
\put(14,58){\vector(1,-1){3.6}}
\put(17,58){\vector(1,-1){3.6}}

\put(2,62){\vector(1,-2){1.8}}
\put(5,62){\vector(1,-2){1.8}}
\put(8,62){\vector(1,-2){1.8}}
\put(11,62){\vector(1,-2){1.8}}
\put(14,62){\vector(1,-2){1.8}}
\put(17,62){\vector(1,-2){1.8}}

\put(2,66){\vector(1,-3){1.2}}
\put(5,66){\vector(1,-3){1.2}}
\put(8,66){\vector(1,-3){1.2}}
\put(11,66){\vector(1,-3){1.2}}
\put(14,66){\vector(1,-3){1.2}}
\put(17,66){\vector(1,-3){1.2}}

\put(2,70){\vector(1,-3){1.2}}
\put(5,70){\vector(1,-3){1.2}}
\put(8,70){\vector(1,-3){1.2}}
\put(11,70){\vector(1,-3){1.2}}
\put(14,70){\vector(1,-3){1.2}}
\put(17,70){\vector(1,-3){1.2}}

\put(2,75){\vector(0,-1){4}}
\put(5,75){\vector(0,-1){4}}
\put(8,75){\vector(0,-1){4}}
\put(11,75){\vector(0,-1){4}}
\put(14,75){\vector(0,-1){4}}
\put(17,75){\vector(0,-1){4}}

\put(2,80){\vector(0,-1){4}}
\put(5,80){\vector(0,-1){4}}
\put(8,80){\vector(0,-1){4}}
\put(11,80){\vector(0,-1){4}}
\put(14,80){\vector(0,-1){4}}
\put(17,80){\vector(0,-1){4}}

\put(2,85){\vector(0,-1){4}}
\put(5,85){\vector(0,-1){4}}
\put(8,85){\vector(0,-1){4}}
\put(11,85){\vector(0,-1){4}}
\put(14,85){\vector(0,-1){4}}
\put(17,85){\vector(0,-1){4}}

%%%%%%%%%%%%%%%%%%%%%%%%%%%%%%%%%%%%%% cross section

\put(0,0){\textbf{The transverse wall}}

\put(60,0){\textbf{The vortex wall}}

\put(27,40){\line(0,1){10}}
\put(27,40){\line(1,0){21}}
\put(48,40){\line(0,1){10}}
\put(27,50){\line(1,0){21}}

\put(28,41){\vector(1,0){3}}
\put(28,43){\vector(1,0){3}}
\put(28,45){\vector(1,0){3}}
\put(28,47){\vector(1,0){3}}
\put(28,49){\vector(1,0){3}}

\put(32,41){\vector(1,0){3}}
\put(32,43){\vector(1,0){3}}
\put(32,45){\vector(1,0){3}}
\put(32,47){\vector(1,0){3}}
\put(32,49){\vector(1,0){3}}

\put(36,41){\vector(1,0){3}}
\put(36,43){\vector(1,0){3}}
\put(36,45){\vector(1,0){3}}
\put(36,47){\vector(1,0){3}}
\put(36,49){\vector(1,0){3}}

\put(40,41){\vector(1,0){3}}
\put(40,43){\vector(1,0){3}}
\put(40,45){\vector(1,0){3}}
\put(40,47){\vector(1,0){3}}
\put(40,49){\vector(1,0){3}}

\put(44,41){\vector(1,0){3}}
\put(44,43){\vector(1,0){3}}
\put(44,45){\vector(1,0){3}}
\put(44,47){\vector(1,0){3}}
\put(44,49){\vector(1,0){3}}

%%%%%%%%%%%%%%%%%%%%%%%% vortex wall

\put(55,15){\line(0,1){71}}
\put(55,15){\line(1,0){20}}
\put(55,86){\line(1,0){20}}
\put(75,15){\line(0,1){71}}

\put(55,28){\line(1,2){20}}
\put(75,28){\line(-1,2){20}}

\put(83,30){\line(0,1){30}}
\put(83,30){\line(1,0){30}}
\put(113,30){\line(0,1){30}}
\put(83,60){\line(1,0){30}}

\put(83,30){\line(1,1){30}}
\put(83,60){\line(1,-1){30}}
\put(90.5,30){\line(1,2){15}}
\put(113,37.5){\line(-2,1){30}}

\thicklines

\put(93,55){\vector(1,0){6}}
\put(105.5,55){\vector(1,-1){3.7}}

\put(103,35){\vector(-1,0){6}}
\put(90.5,35){\vector(-1,1){3.7}}

\put(88,40){\vector(0,1){6}}
\put(88,52.5){\vector(1,1){3.7}}

\put(108,50){\vector(0,-1){6}}
\put(108,37.5){\vector(-1,-1){3.7}}

\put(65,47){\vector(0,-1){5}}
\put(63.5,44){\vector(0,-1){5}}
\put(66.5,44){\vector(0,-1){5}}

\put(62,40){\vector(0,-1){5}}
\put(65,40){\vector(0,-1){5}}
\put(68,40){\vector(0,-1){5}}

\put(60.5,36){\vector(0,-1){5}}
\put(66.5,36){\vector(0,-1){5}}
\put(63.5,36){\vector(0,-1){5}}
\put(69.5,36){\vector(0,-1){5}}

\put(58,30){\vector(0,-1){5}}
\put(63,30){\vector(0,-1){5}}
\put(68,30){\vector(0,-1){5}}
\put(73,30){\vector(0,-1){5}}

\put(58,23){\vector(0,-1){5}}
\put(63,23){\vector(0,-1){5}}
\put(68,23){\vector(0,-1){5}}
\put(73,23){\vector(0,-1){5}}

%%%%%%%%%%%%%%%%%%%5 upper part

\put(65,49){\vector(0,1){5}}

\put(63.5,52){\vector(0,1){5}}
\put(66.5,52){\vector(0,1){5}}

\put(62,56){\vector(0,1){5}}
\put(65,56){\vector(0,1){5}}
\put(68,56){\vector(0,1){5}}

\put(60.5,60){\vector(0,1){5}}
\put(66.5,60){\vector(0,1){5}}
\put(63.5,60){\vector(0,1){5}}
\put(69.5,60){\vector(0,1){5}}

\put(58,66){\vector(0,1){5}}
\put(63,66){\vector(0,1){5}}
\put(68,66){\vector(0,1){5}}
\put(73,66){\vector(0,1){5}}

\put(63,73){\vector(0,1){5}}
\put(68,73){\vector(0,1){5}}
\put(73,73){\vector(0,1){5}}
\put(58,73){\vector(0,1){5}}

\put(58,80){\vector(0,1){5}}
\put(63,80){\vector(0,1){5}}
\put(68,80){\vector(0,1){5}}
\put(73,80){\vector(0,1){5}}

%%%%%%%%%%%%%%%%%%%%%%%%%5 the lateral parts

\put(62,49.5){\vector(0,1){3}}
\put(62,46.5){\vector(0,-1){3}}
\put(68,49.5){\vector(0,1){3}}
\put(68,46.5){\vector(0,-1){3}}

\put(62,48){\circle*{0.7}}
\put(68,48){\circle*{0.7}}
\put(58,48){\circle*{0.7}}
\put(72,48){\circle*{0.7}}

\put(58,46){\vector(0,-1){3}}
\put(58,41.5){\vector(0,-1){4.5}}
\put(58,50){\vector(0,1){3}}
\put(58,55){\vector(0,1){4.5}}

\put(72,46){\vector(0,-1){3}}
\put(72,41.5){\vector(0,-1){4.5}}
\put(72,50){\vector(0,1){3}}
\put(72,55){\vector(0,1){4.5}}

%%%%%%%%%%%%%%%%%%%%%%%%%%%% the cross section

\put(5,7){\textbf{ Figure 1.}}\\
\end{picture}
\\
\\

A distinctive crossover has been observed between the two different modes, which is expected to occur at a critical diameter of the wire. It has been suggested that the magnetization switching process can be understood by analyzing the micromagnetics energy minimization problem for different diameters of the cross section. In [\ref{bib:K.Kuhn}] K. K\"uhn studied $180$ degree static domain walls in magnetic wires with circular cross sections. K\"uhn proved that indeed, the transverse mode must occur in thin magnetic wires as was predicted by experimental and numerical analysis before in [\ref{bib:Fo.Sc.Su.Sc.Ts.Di.Fi}] and in [\ref{bib:He.Ki}],
while in thick wires a vortex wall has the optimal energy scaling. Some of the results proven by K. K\"uhn for thin wires has been later generalized in [\ref{bib:Davit.dissertation}] to any wires with a bounded, Lipschitz and rotationally symmetric cross sections, see also [\ref{bib:Davit.nanowires}]. Slastikov and Sonnenberg proved the energy $\Gamma$-convergence result in [\ref{bib:SS}] for any $C^1$ cross sections in finite curved wires. It is shown
in [\ref{bib:K.Kuhn}], [\ref{bib:Davit.dissertation}] and [\ref{bib:SS}] that the minimal energy scales like $d^2,$ where $d$ is the diameter of the wire, provided the wire cross section has comparable dimensions. It turns out that if the dimensions of the cross section are not comparable, then the minimal energies decay faster than $d^2$ and a logarithmic term occurs. In this paper we study the minimal energy scaling in infinite thin films, as both sides of the cross section go to zero, but one faster that the other. The minimal appears to scale like $d^2(\ln l-\ln d),$ where $0<d<l$ are the dimensions of the cross section. The paper is organized as follows: In section 2 we make some notations and formulate the main results. In section 3 we prove that for small cross section diameters the magnetostatic energy can ba approximated by a quadratic form in the second and the third components of the magnetization $m.$ In section 4 we prove when the diameter goes to zero the energy minimization problems $\Gamma$-converge to a one-dimensional problem. In section 5 we prove a rate of convergence on the minimal energies as the diameter of the film goes o=to zero. Finally, in section 6 we prove to auxiliary lemmas.

\section{The main results}

Denote $\Omega(l,d)=\mathbb R\times R(l,d)$, where $R(l,d)=[-l, l]\times[-d, d]$ and throughout this work it will be assumed that $0<d\leq l.$ Denote the aspect ratio $c=\frac{d}{l}.$ Consider the energy of micromagnetics without an external field and anisotropy energy, i.e., the energy of an isotropic ferromagnet with the absence of an external field:
$$E(m)=A_{ex}\int_\Omega|\nabla m|^2+K_d\int_{\mathbb R}|\nabla u|^2.$$
By scaling of all coordinates one can reach the situation when $A_{ex}=K_d,$ so we will henceforth assume that $A_{ex}=K_d=1.$ Denote
$$A(\Omega)=\{m\colon\Omega\to\mathbb S^2 \ : \ m\in H_{loc}^1(\Omega), \ E(m)<\infty\},$$
and also let us introduce $180$ degree domain walls
$$\tilde A(\Omega)=\{m\colon\Omega\to\mathbb S^2 \ : \ m-\bar e\in H^1(\Omega)\},$$
where
\begin{equation*}
\bar e(x,y,z) = \left\{
\begin{array}{rl}
(-1,0,0) & \text{if } \ \ x<-1 \\
(x,0,0)  & \text{if } \ \ -1\leq x \leq 1 \\
(1,0,0) & \text{if } \ \ 1<x \\
\end{array} \right.
\end{equation*}
Roughly speaking we are considering the set of all magnetizations that satisfy $\lim_{x\to\pm\infty}m(x,y,z)=\pm \vec e_x$ for all $y$ and $z.$
The target of this paper will be studying the minimal energy scaling and the minimizers in minimization problem
\begin{equation}
\label{minimization problem}
\inf_{m\in \tilde A(\Omega(l,d))}E(m)
\end{equation}

when $l\to 0$ and $c\to 0.$
It turns out that after rescaling the energy by a suitable factor, the new rescaled energies $\Gamma$-converge to a one dimensional energy, e.g. [\ref{bib:K.Kuhn},\ref{bib:Davit.dissertation},\ref{bib:SS}]. Consider sequences of domain-magnetization-energy triples $(\Omega(l_n,d_n), m^n, E(m^n))$ such that $d_n,l_n\to 0$ and $c_n=\frac{d_n}{l_n}\to 0$ as $n$ goes to infinity. Denote for simplicity $\Omega_n=\Omega(l_n,d_n)$, $A_n=A(\Omega(l_n,d_n)),$ $\tilde A_n=\tilde A(\Omega(l_n,d_n)).$
Set $\lambda_n=\frac{1}{c_n|\ln c_n|},$ $\mu_n=\frac{l_nd_n}{\lambda_n}$ and rescale the magnetization $m$ as follows: $\acute m(x,y,z)=m(\lambda_nx,l_ny,d_nz).$ Note that, in contrast to [\ref{bib:K.Kuhn},\ref{bib:SS},\ref{bib:Davit.nanowires}] we rescale $m$ in the $x$ direction as well.

Denote now $\acute E(\acute m^n)=\frac{E(m^n)}{\mu_n}$  and consider the rescaled minimization problems

\begin{equation}
\label{minimization problem, rescaled}
\inf_{m\in \tilde A_n}\acute E(\acute m)
\end{equation}

instead of the original problem

$$\inf_{m\in \tilde A_n}E(m). $$
 The rescaled energy functional will have the form:

$$\acute E(\acute m^n)=\frac{1}{\mu_n}\int_{\Omega(1,1)}\Big(|\partial_x \acute m^n(\xi)|^2+\frac{\lambda_n^2}{l_n^2}|\partial_y \acute m^n(\xi)|^2+\frac{\lambda_n^2}{d_n^2}|\partial_z \acute m^n(\xi)|^2\Big)\ud \xi+\frac{E_{mag}(m^n)}{\mu_n}.$$
The limit(reduced) energy functional $E_0$ turns out to be
\begin{equation*}
 E_0(m) = \left\{
\begin{array}{rl}
4\int_{\mathbb{R}}|\partial_x m|^2\ud x+\frac{4}{\pi}\int_{\mathbb{R}}| m_2|^2\ud x, &\quad \text{if}\quad m_3\equiv 0 \\
+\infty, &  \quad \text{otherwise} \\
\end{array} \right.
\end{equation*}
and the admissible set $A_0$ for the reduced variational problem is
$$A_0=\{m\colon\mathbb{R}\to \mathbb{S}^2 \ | \ m(\pm\infty)=\pm 1\}.$$

The reduced or limit variational problem is to minimize the reduced energy functional $E_0$ over the admissible set $A_0,$ i.e.,
\begin{equation}
\label{minimization problem, reduced}
\inf_{m\in A_0}E_0(m).
\end{equation}

Define furthermore
$$A_0^3=\{m\in A_0\ | \ m_3\equiv 0\}.$$

The equality $\min_{m\in A_0}E_0(m)=\min_{m\in  A_0^3}E_0(m)$ suggests considering the minimization problem
$\min_{m\in A_0^3}E_0(m)$ instead of $\min_{m\in A_0}E_0(m).$
 Next define the notion of convergence of the magnetizations like in [\ref{bib:K.Kuhn},\ref{bib:Davit.dissertation}].
\newtheorem{Theorem}{Theorem}[section]
\newtheorem{Lemma}[Theorem]{Lemma}
\newtheorem{Corollary}[Theorem]{Corollary}
\newtheorem{Remark}[Theorem]{Remark}
\newtheorem{Definition}[Theorem]{Definition}

\begin{Definition}
\label{notion of convergence}
The sequence  $\{m^n\}\subset A(\Omega)$ is said to converge to $m^0\in A(\Omega)$ as $n$ goes to infinity if,
\begin{itemize}
\item[(i)] $\nabla m^n\rightharpoonup\nabla m^0 $ \ \ weakly in \ \ $L^2(\Omega)$
\item[(ii)] $m^n \rightarrow m^0$ \ \ strongly in \ \ $L_{loc}^2(\Omega)$
\end{itemize}

\end{Definition}

\begin{Theorem}[$\Gamma$-convergence]
\label{thm:gamma.convergence}
The reduced variational problem is the $\Gamma$-limit of the full variational problem with respect to the convergence stated in Definition~\ref{notion of convergence}. This amounts to the following three statements:
\begin{itemize}
 \item \textbf{Lower semicontinuity.} If a sequence of rescaled magnetizations $\{\acute m^n\}$ with $m^n\in A_n$ converges to some $m^0\in A(\Omega)$ in the sense of Definition~\ref{notion of convergence} then
    $$ E_0(m^0)\leq\liminf_{n\to\infty} \acute E(\acute m^n)$$
\item \textbf{Recovery sequence.} For every $m^0\in A_0$ and every sequence of pairs $\{(l_n,d_n)\}$ with $l_n,d_n\rightarrow 0,$ $c_n\to 0,$ there exists a sequence $\{m^n\}$ with $m^n\in\tilde A_n$ such that
    \begin{align*}
    &\acute m^n\to m^0 \quad \text{in the sense of Definition~\ref{notion of convergence}}\\
    &E_0(m^0)=\lim_{n\to\infty} \acute E(\acute m^n)
    \end{align*}
\item \textbf{Compactness.} Let  $\{(l_n,d_n)\}$ be such that $l_n,d_n\rightarrow 0$ and $c_n\to 0$. Assume $m^n\in \tilde A_n$ and $\acute E(\acute m^n)\leq C$ for  all $n\in\mathbb N.$ Then there exists a subsequence of $\{m^n\}$ (not relabeled) such that after a translation in the $x$ direction the sequence $\acute m^n$ converges to some $m^0\in A_0^3$ in the sense of Definition~\ref{notion of convergence}.
\end{itemize}
\end{Theorem}

\begin{Corollary}
Due to the above theorem we have
\begin{equation}
\label{conv.of.mins}
\lim_{n\to\infty}\min_{m^n\in\tilde A_n}\acute E(\acute m^n)=\min_{m\in A_0}E(m).
\end{equation}
\end{Corollary}
As will be seen later $\min_{m\in A_0}E(m)=\frac{16}{\sqrt\pi}.$

The next theorem establishes a rate of convergence for (\ref{conv.of.mins}).

\begin{Theorem}[Rate of convergence]
\label{thm:rate.conv}
For sufficiently small $d$ and $c$ the following bound holds:
$$
\Big|\min_{m\in\tilde A}\acute E(\acute m)-\min_{m\in A_0}E_0(m)\Big|\leq \frac{200}{\sqrt{|\ln c|}}+20l.
$$
\end{Theorem}

\section{An approximation of the magnetostatic energy}

Recall that the map $u$ is a weak solution of $\triangle u=\mathrm{div}m $ if and only if
\begin{equation}
\label{main.identity}
\int_{\mathbb{R}^3}\nabla{u}\cdot\nabla{\varphi}=\int_\Omega m\cdot\nabla \varphi \qquad \text{for all} \qquad \varphi\in C_0^{\infty}(\mathbb{R}^3).
\end{equation}
The left hand side of the above equality can be written as a sum volume and surface contributions as:
\begin{equation}
\label{modified main identity}
\int_{\mathbb{R}^3}\nabla{u}\cdot\nabla{\varphi}=-\int_\Omega \mathrm{div}m \cdot\varphi+\int_{\partial{\Omega}} m\cdot n\varphi \qquad \text{for all} \qquad \varphi\in C_0^{\infty}(\mathbb{R}^3),
\end{equation}
where $n$ is the outward unite normal to $\partial\Omega.$

Denoting
$$u_v(\xi)=-\int_\Omega\Gamma(\xi-\xi_1)(\mathrm{div}m)(\xi)\ud \xi\quad\text{and}\quad u_s(\xi)=\int_{\partial\Omega}\Gamma(\xi-\xi_1)(m\cdot n)(\xi)\ud \xi,$$
where $\Gamma(\xi)=\frac{1}{4\pi|\xi|}$ is the Green function in $\mathbb R^3$, we obtain
 \begin{equation}
\label{modified main identity for decompositions}
\int_{\mathbb{R}^3}\nabla{u_v}\cdot\nabla{\varphi}=\int_\Omega v\cdot\varphi,\qquad \int_{\mathbb{R}^3}\nabla{u_s}\cdot\nabla{\varphi}=\int_{\partial{\Omega}}s\cdot\varphi \qquad \text{for all} \qquad \varphi\in C_0^{\infty}(\mathbb{R}^3).
\end{equation}

Denote furthermore
$$E_v=\int_{\mathbb{R}^3}|\nabla u_v|^2,\qquad E_s=\int_{\mathbb{R}^3}|\nabla u_s|^2,\qquad E_{vs}=\int_{\mathbb{R}^3}\nabla u_v\cdot \nabla u_s.$$

Following Kohn and Slastikov as in [\ref{bib:Ko.Sl}] define the average of the magnetization vector over the cross section:
$$\bar m(x,y,z)=\frac{1}{4ld}\int_{R(l,d)}m\ud y\ud z, \qquad  (x,y,z)\in\Omega.$$
Like $m$ we extend $\bar m$ as $0$ outside $\Omega.$
 In this section we prove upper and lower bound on the magnetostatic energy for thin films. We start with the $E_s$ part of the energy.
If the parametrization
\begin{equation*}
\left\{
\begin{array}{rl}
y=y(t),  & \ t\in[0,2]  \\
z=z(t),  & \ t\in[0,2]  \\
\end{array} \right.
\end{equation*}
of $\partial R(l,d)$ is chosen by symmetry so that $y(t+1)=-y(t),$ $z(t+1)=-z(t)$
then Theorem~3.3.5 of [\ref{bib:Davit.dissertation}] delivers a formula for $E_s(m)$ in Fourier space for $m=m(x),$ namely:
\begin{Theorem}
\label{representation for E_s}
For every $m=m(x)\in A(\Omega)$ there holds
\begin{align*}
E_s(m)&=\frac{1}{4\pi^2}\int_{\mathbb R^3}\frac{1}{|k|^2}\Big\{|a|^2|\hat m_2(k_1)|^2+|b|^2|\hat m_3(k_1)|^2\\
&+\bar ab(\hat m_2(k_1)\overline{ \hat m_3(k_1)}+\overline{\hat m_2(k_1)}\hat m_3(k_1)\Big\}\ud k,
\end{align*}
where
\begin{align*}
a(k_2,k_3,\omega)&=-2i\int_0^1 z'(t)\sin(k_2y(t)+k_3z(t))\ud t,\\
b(k_2,k_3,\omega)&=2i\int_0^1 y'(t)\sin(k_2y(t)+k_3z(t))\ud t.
\end{align*}
\end{Theorem}
Observe, that when the cross section is the rectangle $R(l,d)$ then the formula for $E_s$ can be easily simplified in more steps, namely, for any $m=m(x)\in A(\Omega),$ we have the following representation formula
$$E_s(m)=\frac{4}{\pi^2}\int_{\mathbb{R}^3}\frac{\sin^2(ly)\sin^2(dz)}{|\xi|^2}\bigg( \frac{|\widehat m_2(x)|^2}{z^2}+\frac{|\widehat m_3(x)|^2}{y^2}\bigg )\ud \xi.$$

Set now for convenience
$$I(l,d,x)=\int_{\mathbb{R}^2}\frac{\sin^2(ly)\sin^2(dz)}{y^2|\xi|^2}\ud y\ud z,$$
then
$$E_s(m)=\frac{4}{\pi^2}\int_{\mathbb{R}}\Big (I(l,d,x)|\widehat m_3(x)|^2+I(d,l,x)|\widehat m_2(x)|^2\Big )\ud x.$$

 The following functions will play an important role in this work. Denote for any $c>0,$

\begin{equation}
\label{a_c.and.b_c}
 a_c=\frac{c}{2}\int_0^\infty\frac{\sin^2 t}{t^2}\cdot\frac{1-e^{-\frac{2t}{c}}}{t}\ud t,\qquad b_c=a_{\frac{1}{c}}.
 \end{equation}

\begin{Lemma}
\label{lem:estimates.on.I(l,d)}
For any $0<d\leq l,$ we have

\begin{itemize}

\item[(i)] Then $I(d,l,x)\leq  2\pi lda_{c}\quad\text{ and}\quad I(l,d,x)\leq  2\pi ldb_{c}\quad\text{for all}\quad x\in\mathbb R,$

\item[(ii)] $I(d,l,x)\leq \pi ldc(3-\ln c),\quad\text{for all}\quad x\in \mathbb{R},$

\item[(iii)]

$ I(d,l,x)\geq\pi ldc|\ln c|\Big(1-\frac{5}{\sqrt{|\ln c}|}\Big) ,\quad \text{for all}\quad x\in\big[-\frac{1}{l}, \frac{1}{l}\big].$

 \end{itemize}

\end{Lemma}

\begin{proof}

We will use the following two identities, that are well known and be found in most advanced calculus and complex analysis textbooks:
\begin{equation}
\label{identities with sin}
 \int_0^{\infty}\frac{\sin^2t }{t^2} \ud t=\frac{\pi}{2},\qquad \int_0^{\infty} \frac{\sin^2(pt) }{t^2+q^2} \ud t=\frac{\pi}{4q}(1-e^{-2pq}),\qquad p,q>0.
\end{equation}
For any $x\neq 0$ we have by making a change of variables $y\to |x|y$, $z\to |x|z$ and putting $a=l|x|, b=d|x|,$
\begin{align*}
I(l,d,x)&=4\int_0^\infty\int_0^\infty\frac{\sin^2(ly)\sin^2(dz)}{y^2|\xi|^2}\ud y \ud z\\
&=\frac{4}{x^2}\int_0^\infty\int_0^\infty\frac{\sin^2(ay)\sin^2(bz)}{y^2(1+y^2+z^2)}\ud y \ud z.
\end{align*}
Utilizing now the second identity of (\ref{identities with sin}) and making a change of variables $y=\frac{t}{a}$ we obtain

\begin{align*}
 I(l,d,x)&=\frac{\pi}{x^2}\int_0^\infty\frac{\sin^2(ay)}{y^2}\cdot\frac{1-e^{-2b\sqrt{y^2+1}}}{\sqrt{y^2+1}}\ud y\\
&=\frac{2\pi ab}{x^2}\int_0^\infty\frac{\sin^2 t}{t^2}\cdot\frac{1-e^{-\frac{2b}{a}\sqrt{t^2+a^2}}}{\frac{2b}{a}\sqrt{t^2+a^2}}\ud t\\
&= 2 \pi ld\int_0^\infty\frac{\sin^2 t}{t^2}\cdot\frac{1-e^{-\frac{2d}{l}\sqrt{t^2+l^2x^2}}}{\frac{2d}{l}\sqrt{t^2+l^2x^2}}\ud t.
\end{align*}
By the inequality
$$\frac{2l}{d}\sqrt{t^2+d^2x^2}\ge \frac{2l}{d}t=\frac{2t}{c}$$
and the fact that the function $\frac{1-e^{-t}}{t}$ decreases over $(0,+\infty)$ we get
\begin{equation}
\label{ineq:I(l,d,x)leqa_c}
I(d,l,x)\leq 2\pi ld\int_0^{+\infty}\frac{\sin^2 t}{t^2}\cdot\ \frac{1-e^{-\frac{2t}{c}}}{ \frac{2t}{c}}\ud t=2\pi lda_c.
 \end{equation}
 Similarly we have $I(l,d,x)\leq2\pi ld b_c.$

For $(ii)$ we have that $I(d,l,x)\leq I_1+I_2+I_3,$
 where
\begin{align*}
I_1&=\pi ldc\int_0^{c}\frac{\sin^2 t}{t^2}\cdot\ \frac{1-e^{-\frac{2t}{c}}}{t}\ud t,\\
I_2&=\pi ldc\int_{c}^1\frac{\sin^2 t}{t^2}\cdot\ \frac{1-e^{-\frac{2t}{c}}}{t}\ud t,\\
I_3&=\pi ldc\int_1^{+\infty}\frac{\sin^2 t}{t^2}\cdot\ \frac{1-e^{-\frac{2t}{c}}}{t}\ud t.
\end{align*}
It is clear that

\begin{align*}
I_1&= 2\pi ld\int_0^{c}\frac{\sin^2 t}{t^2}\cdot\frac{1-e^{-\frac{2t}{c}}}{ \frac{2t}{c}}\ud t\leq 2\pi ld \int_0^{c}\ud t=2\pi ldc,\\
I_2&\leq \pi ldc\int_{c}^1\frac{1}{t}\ud t=-\pi ldc\ln c,\\
I_3&\leq \pi ldc\int_1^{+\infty}\frac{\sin^2 t}{t^2}\ud t\leq \pi ldc\int_1^{+\infty}\frac{1}{t^2}\ud t=\pi ldc.
\end{align*}
Therefore we obtain $I(d,l,x)\leq \pi ldc(3-\ln c)$ and $(ii)$ is proved.

To get a lower bound on $I(d,l,x)$ we note that the main contribution to the integral comes from the interval $[c,1].$ The idea is replacing in the previous argument $[c,1]$ by $[c^{1-\epsilon},c^{\epsilon}]$ where $\epsilon$ is a small positive number yet to be chosen. Assume $\epsilon<\frac{1}{3}$ and  $x\in\big[-\frac{1}{l},\frac{1}{l}\big].$ For any  $t\in[c^{1-\epsilon},c^{\epsilon}]$
we have
$$\frac{2l}{d}\sqrt{t^2+x^2d^2}\geq\frac{2t}{c}\geq 2c^{-\epsilon},$$
and
$$\sqrt{t^2+x^2d^2}\leq t+|x|d\leq t+\frac{d}{l}=t+c,$$
hence
\begin{equation}
\label{1}
I(d,l,x)\ge \pi ldc\int_{c^{1-\epsilon}}^{c^{\epsilon}}\frac{\sin^2 t}{t^2}\cdot \frac{1-e^{-2c^{-\epsilon}}}{t+c}\ud t.
\end{equation}
If we choose now $\epsilon=\frac{1}{\sqrt{|\ln c|}}$ then $c^{\epsilon}\to 0,$ thus we get,
$$1-e^{-2c^{-\epsilon}}>1-\frac{1}{2c^{-\epsilon}}=1-\frac{c^\epsilon}{2},$$

 $$\frac{\sin^2 t}{t^2}\geq \frac{\Big(t- \frac{t^3}{6}\Big)^2}{t^2}\geq 1-t^2, \qquad t\in[0,c^{\epsilon}].$$

Thus we obtain by (\ref{1}),
\begin{align*}
I(d,l,x)&\geq \pi ldc(1-c^{2\epsilon})\Big(1-\frac{c^\epsilon}{2}\Big) \int_{c^{1-\epsilon}}^{c^{\epsilon}}\frac{1}{t+c}\ud t\\
&\geq\pi ldc (1-2c^{\epsilon})\big(\ln(c+c^{\epsilon})-\ln(c+c^{1-\epsilon})\big).
\end{align*}
It is clear that
\begin{align*}
\ln(c+c^{1-\epsilon})&=\ln c+\ln(1+c^{-\epsilon})\\
&\leq \ln c+\ln(2c^{-\epsilon})\\
&\leq (1-2\epsilon)\ln c,
\end{align*}
and
$$\ln(c+c^{\epsilon})\ge \ln c^{\epsilon}=\epsilon \ln c,$$
$$1-2c^{\epsilon}=1-2e^{\epsilon\ln c}>1-2e^{-\frac{1}{\epsilon}}>1-2\epsilon.$$
Concluding we obtain
\begin{align*}
I(d,l,x)&\geq \pi (1-2c^{\epsilon})(1-3\epsilon)ldc|\ln c|\\
&\geq \pi ldc|\ln c|(1-5\epsilon)\\
&=\pi ldc|\ln c|\bigg(1-\frac{5}{\sqrt{|\ln c}|}\bigg).
\end{align*}

\end{proof}
\begin{Corollary}
\label{scaling of a_c}

We have that
$$\lim_{c\to 0}\frac{a_c}{c|\ln c|}=\frac{1}{2}$$
\end{Corollary}

\begin{proof}
The proof follows from $(ii)$ and $(iii)$ parts of the above lemma.
\end{proof}

It is straightforward to see that due to the symmetry of the cross section $R(l,d)$ one has $E_{vs}(m)=0$ for all $m=m(x)\in A(\Omega).$ We estimate now the volume contribution $E_v$ to $E_{mag}.$

\begin{Lemma}
\label{lem:E_v.upperbound}
For any $0<d\leq l$ and $m=m(x)\in A$ the following bound holds:

\begin{equation}
 E_{v}(m)\leq M_m\Big(l^2d^2+ld^2\Big(1+\ln{\frac{l}{d}}\Big)\Big),
\end{equation}

where $M_m$ is a constant depending on the magnetization $m.$
\end{Lemma}

\begin{proof}

By density argument (\ref{modified main identity for decompositions}) holds for $\varphi=u_v$ thus,
$$E_v(m)=\int_{\mathbb{R}^3}|\nabla u_v|^2=-\int_{\Omega}\mathrm{div}m\cdot u_v=\int_{\Omega}\int_{\Omega}\Gamma(\xi-\xi_1)\mathrm{div}m(\xi)\mathrm{div}m(\xi_1)\ud\xi\ud\xi_1.$$
For any $m=m(x)\in A$ we have $\mathrm{div}m=\partial_x m_1(x),$ thus
$$E_v(m)=\frac{1}{4\pi}\int_{\Omega}\int_{\Omega}\frac{\partial_x m_1(x)\partial_x m_{1}(x_1)}{|\xi-\xi_1|}\ud\xi\ud\xi_1$$ where $\xi=(x,y,z)$ and $\xi_1=(x_1,y_1,z_1).$
We have by integration by parts
\begin{align*}
\int_{\mathbb{R}}\frac{\partial_x m_x(x)}{|\xi-\xi_1|}\ud x&=\int_{-\infty}^0\frac{\ud m^\ast(x)}{|\xi-\xi_1|}+\int_0^{+\infty}\frac{\ud m^\ast(x)}{|\xi-\xi_1|}\\
&=\frac{2}{\sqrt{x_1^2+(y-y_1)^2+(z-z_1)^2}}-\int_{\mathbb{R}}\frac{(x-x_1) m^\ast(x)}{|\xi-\xi_1|^3}\ud x,
\end{align*}
where
 $$
m^\ast(x) = \left\{
\begin{array}{rl}
m_1(x)+1 & \text{if } \ \ x\leq0 \\
 m_1(x)-1 & \text{if } \ \ x>0. \\
\end{array} \right.
$$
Then it has been shown in [\ref{bib:Davit.dissertation}] that $m^\ast(x)\in L^2(\mathbb R)$ and we can estimate
$E_v(m)\leq I_1+I_2,$ where

\begin{align*}
I_1&=\frac{1}{2\pi}\int_{R(l,d)}\int_{\Omega}\frac{|\partial_x m_1(x_1)|} {\sqrt{x_1^2+(y-y_1)^2+(z-z_1)^2}}\ud \xi_1\ud y\ud z\\
I_2&=\int_{\Omega}\int_{\Omega}\frac{|\partial_x m_1(x_1)m^{\ast}(x)|}{|\xi-\xi_1|^2}\ud\xi\ud\xi_1.
\end{align*}
We have furthermore
\begin{align*}
\int_{\mathbb{R}}\frac{|\partial_x m_1(x_1)|}{\sqrt{x_1^2+(y-y_1)^2+(z-z_1)^2}}&\ud x_1\\
&\leq\frac{1}{2}\int_{\mathbb{R}}\Big(|\partial_x m_1(x_1)|^2+\frac{1}{x_1^2+(y-y_1)^2+(z-z_1)^2}\Big)\ud x_1\\
&=\frac{1}{2}\|\partial_x m_1\|_{L^2(\mathbb R)}^2+\frac{\pi}{2\sqrt{(y-y_1)^2+(z-z_1)^2}}.
\end{align*}
Recall now Lemma~A2 from [\ref{bib:Davit.nanowires}], which asserts that for any point $(y_1,z_1)\in\mathbb R^2$ one has
\begin{equation}
\label{green.bound}
\int_{R(l,d)}\frac{1}{\sqrt{(y-y_1)^2+(z-z_1)^2}}\ud y\ud z\leq 10d\left(1+\ln\frac{d}{l}\right).
\end{equation}
Thus we obtain for $I_1,$
\begin{align*}
I_1&\leq \frac{4}{\pi}\|\partial_x m_1\|_{L^2(\mathbb R)}^2l^2d^2+\frac{1}{4}\int_{R(l,d)}\int_{R(l,d)}
\frac{1}{\sqrt{(y-y_1)^2+(z-z_1)^2}}\ud y_1\ud z_1\ud y\ud z\\
&\leq\frac{4}{\pi}\|\partial_x m_1\|_{L^2(\mathbb R)}^2l^2d^2+10ld^2\Big(1+\ln\frac{l}{d}\Big).
\end{align*}
By making a change of variables $\xi_2=\xi_1-\xi$ and utilizing again (\ref{green.bound}) we can estimate,
\begin{align*}
I_2&=\int_{\Omega}\int_{\mathbb{R}\times[-l-y,l-y]\times[-d-z,d-z]}\frac{|m^{\star}(x)|\cdot|\partial_xm_1(x_2+x)|}
{|\xi_2|^2}\ud\xi_2\ud \xi\\
&\leq\frac{1}{2}\int_{R(l,d)}\int_{\mathbb{R}\times[-l-y,l-y]\times[-d-z,d-z]}\int_{\mathbb R}\frac{|m^{\ast}(x)|^2+|\partial_xm_1(x_2+x)|^2}{|\xi_2|^2}\ud x\ud\xi_2\ud y\ud z\\
&=2ld(\|m^{\ast}\|_{L^2(\mathbb R)}^2+\|\partial_x m_1\|_{L^2(\mathbb R)}^2)
\int_{\mathbb{R}\times[-l-y,l-y]\times[-d-z,d-z]}\frac{\ud \xi_2}{|\xi_2|^2}\\
&=2\pi ld(\|m^{\ast}\|_{L^2(\mathbb R)}^2+\|\partial_x m_1\|_{L^2(\mathbb R)}^2)
\int_{R(l,d)}\frac{1}{\sqrt{(y_1-y)^2+(z_1-z)^2}}\ud y_1\ud z_1\\
&\leq 20\pi ld^2\Big(1+\ln{\frac{l}{d}}\Big)(\|m^{\ast}\|_{L^2(\mathbb R)}^2+\|\partial_x m_1\|_{L^2(\mathbb R)}^2).
\end{align*}
The summary of the estimates on $I_1$ and $I_2$ completes the proof.

\end{proof}

\section{The convergence of the energies}
Consider a sequence od domain-magnetization-energy triples $\{(\Omega_n, m^n, E(m^n) )\}$ where $\Omega_n=\mathbb R\times R(l_n,d_n),$ $m^n\in \tilde A_n=\tilde A(\Omega_n)$ and $l_n,c_n\to0.$
Lemma~\ref{lem:estimates.on.I(l,d)} suggests that for sufficiently big $n$ one can formally write for any $m=m(x),$
$$E_s(m^n)\approx \frac{8}{\pi}l_nd_na_{c_n}\int_{\mathbb{R}}|m_2^n(x)|^2\ud x+\frac{8}{\pi}l_nd_nb_{c_n}\int_{\mathbb{R}}|m_3^n(x)|^2\ud x$$

Next, Lemma~\ref{lem:a_c} asserts that $a_{c_n}$ scales like $c_n\ln c_n$ and $b_{c_n}\to \frac{\pi}{2}.$ Furthermore, by Lemma~\ref{lem:E_v.upperbound}, for a fixed $m^n=m^n(x)$ the summand $E_v(m^n)$ decays at least like $l_nd_n^2\ln^2 \frac{l_n}{d_n}$. Rescaling the magnetizations  $\acute m^n(x,y,z)=m^n(\lambda_nx,l_ny,d_nz)$ we can rewrite the exchange energy for all $m^n(x)\in A_n$ as
$$E_{ex}(m^n(x))=\frac{l_nd_n}{\lambda_n}\int_{\Omega(1,1)}\Big(|\partial_x \acute m^n(x)|^2+\frac{\lambda_n^2}{l_n^2}|\partial_y \acute m^n(x)|^2+\frac{\lambda_n^2}{d_n^2}|\partial_z \acute m^n(x)|^2\Big)\ud x,$$
and it is clear that $\acute m^n\colon \Omega(1,1)\to\mathbb{S}^2.$

Thus one would expect that for sufficiently big $n$ the approximation holds
$$E_{ex}(m^n(x))\approx \frac{4l_nd_n}{\lambda_n}\int_{\mathbb{R}}|\partial_x m^n(x)|^2\ud x $$
 and
$$E_s(m^n(x))\approx \frac{4}{\pi}l_nd_nc_n|\ln c_n|\lambda_n\int_{\mathbb{R}}(|m_2^n(x)|^2+\frac{\pi}{c_n|\ln c_n|}|m_3^n(x)|^2)\ud x.$$

This calculation suggests that the coefficients \  $\frac{l_nd_n}{\lambda_n}$ \ and \ $l_nd_nc_n|\ln c_n|\lambda_n$ \  should be taken equal and they will both be the scaling of \ $E(m^n).$ This leads to $\lambda_n=\frac{1}{\sqrt{c_n|\ln c_n|}}.$

\textbf{Proof of Thereom~\ref{thm:gamma.convergence}.}\\
\textbf{Lower semicontinuity.} One can without loss of generality assume that $\acute E(\acute m^n)\leq M$ for some $M>0$ and all $n\in\mathbb{N}.$
Following Kohn and Slastikov [\ref{bib:Ko.Sl}] let us prove that
\begin{equation}
\label{E_magbarm,m}
\liminf_{n\to\infty} \frac{E_{mag}(m^n)}{\mu_n}=\liminf_{n\to\infty} \frac{E_{ mag}(\bar m^n)}{\mu_n}.
\end{equation}
By the Poincar\'e inequality we have
$$\int_{\Omega}|m-\bar m|^2\leq C(d^2+l^2)\int_{\Omega}|\nabla m|^2\leq C(d^2+l^2)E(m).$$

Owing now to the third inequality in Lemma~\ref{lem:ineq.m1.m2} and the above inequality we have
\begin{align}
\label{Emag.bar}
|E_{mag}(m^n)-E_{mag}(\bar m^n)|\leq M_1\mu_n\sqrt{l_n^2+d_n^2},
\end{align}
for some $M_1,$ which implies (\ref{E_magbarm,m}).
Let now $\{q_n\}$ be a sequence with $0<q_n<1$ yet to be defined. We have by the Plancherel equality,
\begin{align*}
q_n\frac{E_{ex}(m^n)}{\mu_n}&\geq\frac{q_n}{\mu_n}\int_{\Omega_n}|\partial_x \bar m^n(\xi)|^2\ud\xi\\
 &=4\frac{q_nl_nd_n}{\mu_n}\int_{\mathbb{R}}|\widehat{\partial_x \bar m^n}(x)|^2\ud x\\
&=4\frac{q_nl_nd_n}{\mu_n}\int_{\mathbb{R}}|x\cdot \widehat{\bar m^n}(x)|^2\ud x\\
&\geq\frac{4q_nd_n}{l_n\mu_n}\int_{\mathbb{R}\setminus[-\frac{1}{l_n},\frac{1}{l_n}]}(|\widehat{\bar m_2^n}(x)|^2+|\widehat{\bar m_3^n}(x)|^2)\ud x,
\end{align*}
and according to part $(iii)$ of Lemma~\ref{lem:estimates.on.I(l,d)} we have for big $n$ as well
$$\frac{E_s(\bar m^n)}{\mu_n}\geq \frac{4}{\pi \mu_n} l_nd_nc_n|\ln c_n|\left(1-\frac{5}{\sqrt{|\ln c_n|}}\right)\int_{-\frac{1}{l_n}}^{\frac{1}{l_n}}\Big(\frac{1}{|\ln c_n|}|\widehat{\bar m_2^n}(x)|^2+|\widehat{\bar m_3^n}(x)|^2\Big)\ud x.$$
Now choose $q_n$ so that
$$\frac{4}{\pi \mu_n} l_nd_nc_n|\ln c_n|\left(1-\frac{5}{\sqrt{|\ln c_n|}}\right)=\frac{4q_nd_n}{l_n\mu_n}, $$
or
$$q_n=\frac{1}{\pi} l_nd_nc_n|\ln c_n|\left(1-\frac{5}{\sqrt{|\ln c_n|}}\right),$$
and it is clear that $q_n\to 0.$

Applying now the obtained inequalities, (\ref{E_magbarm,m}) and the convergence $\nabla \acute m^n\rightharpoonup\nabla m^0$ in $L^2(\Omega(1,1))$ we obtain

\begin{align*}
&\liminf_{n\to\infty}\frac{E(m^n)}{\mu_n}\geq \liminf_{n\to\infty}(1-q_n)\int_{\Omega(1,1)}|\partial_x \acute m^n|^2\ud\xi+\liminf_{n\to\infty}q_n\frac{E_{ex}(m^n)}{\mu_n}+\liminf_{n\to\infty}\frac{E_{mag}(m^n)}{\mu_n}\\
&= \liminf_{n\to\infty}(1-q_n)\int_{\Omega(1,1)}|\partial_x \acute m^n|^2\ud\xi+\liminf_{n\to\infty}q_n\frac{E_{ex}(m^n)}{\mu_n}+\liminf_{n\to\infty}\frac{E_{mag}(\bar m^n)}{\mu_n}\\
&\geq\liminf_{n\to\infty}(1-q_n)\int_{\Omega(1,1)}|\partial_x \acute m^n|^2\ud\xi+\liminf_{n\to\infty}q_n\frac{E_{ex}(m^n)}{\mu_n}+\liminf_{n\to\infty}\frac{E_s(\bar m^n)}{\mu_n}\\
&\geq4\int_{\mathbb{R}}|\partial_x m^0|^2+\liminf_{n\to\infty}\frac{4}{\pi \mu_n}l_nd_nc_n|\ln c_n|\left(1-\frac{5}{\sqrt{|\ln c_n|}}\right)\int_{\mathbb{R}}(|\bar m_2^n|^2+|\ln c_n||\bar m_3^n|^2)\\
 &=4\int_{\mathbb{R}}|\partial_x m^0|^2\ud x+\frac{4}{\pi}\liminf_{n\to\infty}\frac{1}{\lambda_n}\int_{\mathbb{R}}(|\bar m_2^n(x)|^2+|\ln c_n||\bar m_3^n(x)|^2)\ud x.
\end{align*}
It is then standard to prove that the convergence $\acute m^n\to m^0$ in $L_{loc}^2(\Omega(1,1))$ implies
$$
\liminf_{n\to\infty}\frac{1}{\lambda_n}\int_{\mathbb{R}}|\bar m_2^n(x)|^2\geq \int_{\mathbb{R}}|m_2^0(x)|^2\quad\text{and}\quad
\liminf_{n\to\infty}\frac{1}{\lambda_n}\int_{\mathbb{R}}|\bar m_3^n(x)|^2\geq \int_{\mathbb{R}}|m_3^0(x)|^2,
$$
thus since $|\ln c_n|\to\infty$ we conclude that
$$\liminf_{n\to\infty}\frac{E(m^n)}{\mu_n}\geq E_0(m^0).$$
$\textbf{Recovery sequence.}$ Let us prove that the sequence $m^n(x),$ where
 $$m^n(\lambda_nx,y,z )=m^0(x)\quad\text{ if }\quad  \xi\in\Omega(l_n,d_n)\quad \text{ and}\quad  m^n(\xi)=0\quad \text{ if}\quad  \xi\in\mathbb{R}^3\setminus\Omega(l_n,d_n)$$
  satisfies the required condition. If $m_3^0$ is not identically zero, then $E_0(m^0)=\infty$ and
  due to the \textit{lower semi-continuity} part of the foregoing theorem we have that
  $E_0(m^0)\leq\liminf_{n\to\infty}\acute E_n(m^n),$ thus the proof follows. Assume now that $m_3^0\equiv 0.$
 It remains to only prove the reverse inequality $\limsup_{n\to\infty}\acute E_n(\acute m^n)\leq E_0(m^0).$
  It is clear that
$$E(m^n)=4\mu_n\int_{\mathbb{R}}|\partial_x m^0|^2\ud x+E_{mag}(m^n).$$

Due to Lemma~\ref{ineq:I(l,d,x)leqa_c} and the Plancherel equality we have,

$$E_s(m^n)\leq \frac{4}{\pi} l_nd_nc_n(|\ln c_n|+3)\int_{\mathbb{R}}|m_2^0(x)|^2 \ud x,$$
thus

$$\limsup_{n\to\infty}\frac{E_s(m^n)}{\mu_n}\leq \frac{4}{\pi}\int_{\mathbb{R}}|m_2^0(x)|^2 \ud x.$$

We have furthermore by Lemma~\ref{lem:E_v.upperbound} that
$$\limsup_{n\to\infty}\frac{E_v(m^n)}{\mu_n}=0,$$
 thus combining all the obtained inequalities for the energy summands we discover
 $$\limsup_{n\to\infty}\acute E_n(\acute m^n)\leq E_0(m^0).$$
 $\textbf{Compactness.}$  The inequality $E(m^n)\leq C\mu_n$ implies

 $$\int_{\Omega(1,1)}|\partial_x \acute m^n|^2\leq C,\quad \int_{\Omega(1,1)}|\partial_y \acute m^n|^2\leq C\frac{l_n^2}{\lambda_n^2}\quad \text{and}\quad \int_{\Omega(1,1)}|\partial_z \acute m^n|^2\leq C\frac{d_n^2}{\lambda_n^2},$$
 thus a subsequence (not relabeled) of $\{\nabla \acute m^n\}$ has a weak limit $f=f(x)$ in $L^2(\Omega(1,1)).$
 On the other hand $\acute m^n$ has a unit length pointwise, thus a subsequence (not relabeled) of $\{\acute m^n\}$ has a strong local limit $m^0$ in
 $L^2(\Omega(1,1)).$ It is then straightforward to show that $m^0$ is weakly differentiable with $f=\nabla m^0,$ thus $\{\acute m^n\}$ converges
 to $m^0$ in the sense of Definition~\ref{notion of convergence}. It has been proven in [\ref{bib:Davit.nanowires}], that actually one can translate the subsequence $\{\acute m^n\}$ in the $x$ direction so that the limit $m^0$ satisfies $m^0(\pm\infty)=\pm 1.$ 
 Finally owing to the \textit{lower semi-continuity} part of the lemma we discover $E_0(m^0)\leq \liminf \acute E(\acute m^n)\leq C<\infty,$ thus $m_3^0\equiv 0,$ i.e., $m^0\in A_0^3.$

\section{The rate of convergence}
\label{sec:energy.scaling.1}

Recall that for any $\alpha>0$ one can explicitly determine the minima of the energy functional e.g., [\ref{bib:K.Kuhn},\ref{bib:Davit.dissertation},\ref{bib:Davit.nanowires}],
 $$E_{\alpha}(m)=\int_{\mathbb{R}}|\partial_x m(x)|^2\ud x+\alpha\int_{\mathbb{R}}(| m_y(x)|^2+|m_z(x)|^2)\ud x$$
 in the admissible set
 $$A_0=\{m\colon \mathbb{R}\to\mathbb{R}^3 \ : \ |m|=1, m(\pm\infty)=\pm 1\}.$$ The minimizer is given by the formula

 \begin{equation}
 \label{limit problem minimizer}
 m=m^{\alpha,\beta}=\bigg(\frac{e^{2\sqrt{\alpha}x}\cdot\beta^2-1}{e^{2\sqrt{\alpha}x}\cdot\beta^2+1},\  \frac{2\beta e^{\sqrt{\alpha}x}}{e^{2\sqrt{\alpha}x}\cdot\beta^2+1}\cos\theta,\ \frac{2\beta e^{\sqrt{\alpha}x}}{e^{2\sqrt{\alpha}x}\cdot\beta^2+1}\sin\theta\bigg),
 \end{equation}
 where $\beta\in\mathbb R.$
 Note that the minimal value of the energy does not depend on $\theta,$ i.e., it is invariant under rotations in the cross section plane, and for a fixed $\theta$ any minimizer can be obtained from $m^{\alpha}:=m^{\alpha,1}$ by translation in the $x$ direction. The minimizer $m^{\alpha}$ satisfies $m_x^{\alpha}(0)=0.$ The minimal value of $E_\alpha$ in $A_0$ will be $4\sqrt{\alpha}.$ Therefore, due to the fact $m^0\in A_0^3$, the minimizers $m^0$ of $E_0$ have the form

\begin{equation}
\label{lim.prob.min2}
m^0=\bigg(\frac{e^{\frac{2x}{\sqrt\pi}}\cdot\beta^2-1}{e^{\frac{2x}{\sqrt\pi}}\cdot\beta^2+1},\  \frac{2\beta e^{\frac{2x}{\sqrt\pi}}}{e^{\frac{2x}{\sqrt\pi}}\cdot\beta^2+1},\ 0\bigg).
\end{equation}
The minimal value of $E_0$ is $\frac{16}{\sqrt\pi}.$ \\

\textbf{Proof of Theorem~\ref{thm:rate.conv}}  We need to get accurate lower and upper bounds on $E(m)$. For an upper bound we choose the recovery sequence $m(x,y,z)=m^0(\frac{x}{\lambda_n}),$ where $m_3\equiv 0$ and $m^0$ is a minimizer of the energy functional
$$E_0(m)=4\int_{\mathbb R}|\partial_xm|^2\ud x+\frac{4}{\pi}\int_{\mathbb R}(|m_2(x)|^2+|m_3(x)|^2)\ud x.$$

Due to Lemma~\ref{lem:estimates.on.I(l,d)} we have for big $n$
$$E(m^0)\leq\frac{4l_nd_n}{\lambda_n}\int_{\mathbb R}|\partial_xm^0|^2\ud x+\frac{4l_nd_nc_n(3-\ln c_n)}{\pi}\int_{\mathbb R}|m_2^0(x)|^2\ud x+
E_v(m^0).$$

Next, due to Lemma~\ref{lem:E_v.upperbound} we get for big $n,$
\begin{align*}
\frac{E(m)}{\mu_n}&\leq4E_0(m)+\frac{12}{\pi |\ln c_n|}\int_{\mathbb R}|m_z^0(x)|^2\ud x+2M_{m^0}d_n\lambda_n(1-\ln c_n)\\
&\leq \frac{16}{\sqrt\pi}+\frac{10}{|\ln c_n|}+2\sqrt{l_nd_n|\ln c_n|},
\end{align*}
thus the minimal energy satisfies the inequality
\begin{equation}
\label{E.min.upeer.bound}
\frac{\min_{m\in\tilde A_n}E(m)}{\mu_n}- \frac{16}{\sqrt\pi}\leq\frac{10}{|\ln c_n|}+2\sqrt{l_nd_n|\ln c_n|}.
\end{equation}

Assume now $m\in\tilde A_n$ is an energy minimizer in $\Omega_n.$
We have that $I(l_n,d_n,x)\geq I(d_n,l_n,x)$, thus by Lemma~\ref{lem:estimates.on.I(l,d)} we have

$$E_{mag}(\bar m)\geq\frac{4}{\pi} l_nd_nc_n|\ln c_n|\Big(1-\frac{5}{\sqrt{|\ln c_n|}}\Big)
\int_{-\frac{1}{l_n}}^{\frac{1}{l_n}}(|\widehat{\bar m_2}|^2+|\widehat{\bar m_3}|^2)\ud x.$$
According to (\ref{E.min.upeer.bound}) we have for big $n,$
$$
\frac{\min_{m\in\tilde A_n}E(m)}{\mu_n}\leq\frac{16}{\sqrt{\pi}}+1<11.
$$
We have furthermore for big $n$ that
\begin{align*}
\int_{\mathbb R\setminus[-\frac{1}{l_n}, \frac{1}{l_n}]}(|\widehat{\bar m_2}|^2+|\widehat{\bar m_3}|^2)\ud x&\leq
l_n^2\int_{\mathbb R}(|x\cdot\widehat{\bar m_2}|^2+|x\cdot\widehat{\bar m_3}|^2)\ud x\\
&=l_n^2\int_{\mathbb R}(|\partial_x\bar m_2|^2+|\partial_x\bar m_3|^2)\ud x\\
&\leq \frac{l_n}{4d_n}\int_{\Omega_n}(|\partial_x m_2|^2+|\partial_x m_3|^2)\ud x\\
&\leq\frac{l_nE_{ex}( m)}{4d_n}\\
&\leq\frac{11l_n\mu_n}{4d_n},
\end{align*}

thus

$$\frac{4}{\pi}l_nd_nc_n|\ln c_n|\int_{\mathbb R\setminus[-\frac{1}{l_n}, \frac{1}{l_n}]}(|\widehat{\bar m_2}|^2+|\widehat{\bar m_3}|^2)\ud x\leq \frac{11}{\pi}l_n^2c_n|\ln c_n|\mu_n,$$

therefore we obtain
\begin{equation}
\label{E.mag.lower.bdd}
 E_{mag}(\bar m)\geq\frac{4}{\pi} l_nd_nc_n|\ln c_n|\Big(1-\frac{5}{\sqrt{|\ln c_n|}}\Big)
\int_{\mathbb R}(|\bar m_2|^2+|\bar m_3|^2)\ud x-\frac{11}{\pi}l_n^2c_n|\ln c_n|\mu_n
\end{equation}

It is straightforward to see using the definition of the average that
$$\int_{\Omega_n}(|m_2|^2+|m_3|^2)-\int_{\Omega_n}(|\bar m_2|^2+|\bar m_3|^2)=\int_{\Omega_n}(|m_2-\bar m_2|^2+|m_3-\bar m_3|^2),$$
thus by the Poincar\'e inequality we get for big $n,$
\begin{equation}
\label{11}
\int_{\Omega_n}(|m_2|^2+|m_3|^2)\leq \int_{\Omega_n}(|\bar m_2|^2+|\bar m_3|^2)+11C\mu_n(l_n^2+d_n^2).
\end{equation}

Next, due to the estimate (\ref{E_magbarm,m}) we have for big $n$ that

$$E_{mag}(m)\geq E_{mag}(\bar m)-M_1\mu_n\sqrt{d_n^2+l_n^2},$$
where $M_1=11C$ and $C$ is the Poincer\'e constant for $R(l_n,d_n)$.
Combining now the last inequality with \ref{E.mag.lower.bdd} and (\ref{11}), for bin $n$ we discover

\begin{equation}
\label{E.mag.lower.2}
E_{mag}(m)\geq \frac{4}{\pi} l_nd_nc_n|\ln c_n|\Big(1-\frac{5}{\sqrt{|\ln c_n|}}\Big)
\int_{\mathbb R}(|m_2|^2+|m_3|^2)\ud x-\frac{11}{\pi}l_n^2c_n|\ln c_n|\mu_n-12C\mu_n\sqrt{l_n^2+d_n^2}.
\end{equation}

For the whole energy we obtain for big $n,$
\begin{align*}
\frac{E(m)}{ \mu_n }\geq4\Big(1-\frac{5}{\sqrt{|\ln c_n|}}\Big)\Big(\int_{\Omega(1,1)}(|\partial_x \acute m|^2\ud\xi+\frac{1}{\pi}\int_{\Omega(1,1)}(|\acute m_2|^2+|\acute m_3|^2)\ud \xi\Big)-20Cl_n.
\end{align*}
It has been shown in [\ref{bib:Davit.nanowires},Lemma 3.3], that if $m\in A_n$ then $\bar m(\pm\infty)=\pm 1,$ thus we have $\acute m(\pm\infty,y,z)=\pm 1$ on a full measure subset $Q$ of $R(1,1).$ Therefore we have for any $(y,z)\in Q$ that
$$\int_{\mathbb R}(|\partial_x \acute m(x,y,z)|^2\ud x+\frac{1}{\pi}\int_{\mathbb R}(|\acute m_2(x,y,z)|^2+|\acute m_3(x,y,z)|^2)\ud x\geq \frac{4}{\sqrt \pi},$$
which gives
$$\int_{\Omega(1,1)}(|\partial_x \acute m|^2\ud\xi+\frac{1}{\pi}\int_{\Omega(1,1)}(|\acute m_2|^2+|\acute m_3|^2)\ud \xi\geq \frac{16}{\sqrt \pi}.$$
Finally we get for the energies,
$$
\frac{E(m)}{\mu_n}-\frac{16}{\sqrt{\pi}}\geq-\frac{200}{\sqrt{|\ln c_n|}}-20Cl_n.
$$
A combination of the last inequality and (\ref{E.min.upeer.bound}) completes the proof.
In conclusion, let us mention that for sufficiently small $d$ and $l$ the minimizer $m$ must have almost the shape of $m^{\alpha,\beta}$ i.e., must be a 
transverse wall.

\appendix
\section{Appendix}

In this section we recall a key inequality and study the function $a_c.$

\begin{Lemma}
\label{lem:ineq.m1.m2}
For any vector fields $m_1,m_2\in M(\Omega)$ with finite energies there holds
$$|E_{mag}(m_1)-E_{mag}(m_2)|\leq \|m_1-m_2\|_{L^2(\Omega)}^2+2\|m_1-m_2\|_{L^2(\Omega)}\sqrt{E_{mag}(m_1)}$$
\end{Lemma}

\begin{proof}
The proof is trivial and can be found in [\ref{bib:K.Kuhn}].
\end{proof}

Consider now $c\to a_c$ as a map from $(0,+\infty)$ to $(0,+\infty).$

\begin{Lemma}
\label{lem:a_c}
The function $a_c$ has the following properties:
\begin{itemize}
\item[(i)]  $a_c$ increases in $(0,+\infty)$
\item[(ii)] $\lim_{c\to 0}\frac{a_c}{c|\ln c|}=\frac{1}{2}$

\item[(iii)] $\lim_{c\to +\infty}a_c=\frac{\pi}{2}.$
\end{itemize}
\end{Lemma}

\begin{proof}
The first property follows from the fact that the function $\frac{1-e^{-t}}{t}$ decreases over $(0, +\infty).$ The second property is Corollary~\ref{scaling of a_c}. Assume now $c\ge 4.$ It is clear that
$$\frac{1-e^{-\frac{2t}{c}}}{\frac{2t}{c}}\ge 1-\frac{t}{c}\quad\text{if}\quad t\in\Big[0,\frac{c}{2}\Big],$$
thus taking into account the inequality $\sqrt{c}\leq \frac{c}{2})$ we discover
$$\frac{1-e^{-\frac{2t}{c}}}{\frac{2t}{c}}\ge 1-\frac{t}{c}\ge 1-\frac{1}{\sqrt{c}}\quad\text{if}\quad t\in[0,\sqrt{c}].$$
Therefore for $a_c$ we have on one hand
$$\liminf_{c\to\infty }a_c\ge \liminf_{c\to\infty } \Big(1-\frac{1}{\sqrt c}\Big)\int_0^{\sqrt c}\frac{\sin^2 t}{t^2}\ud t=\int_0^{+\infty}\frac{\sin^2 t}{t^2}\ud t=\frac{\pi}{2}, $$
but on the other hand
$$a_c\le \int_0^{+\infty}\frac{\sin^2 t}{t^2}\ud t=\frac{\pi}{2} \quad \text{for any}\quad c>0,$$
which achieves the proof.

\end{proof}

\textbf{\large{Acknowledgement}}\\

The present results are part of the author's PhD thesis. The author is thankful to his supervisor Dr. Prof. S. M\"uller for suggesting the topic and many good advices. This work was supported by scholarships by Max-Planck Institute for Mathematics in the Sciences in Leipzig, Germany, and HCM for Mathematics in Bonn, Germany.

\end{document}